\documentclass{elsart}
\usepackage{amssymb}

\input{tcilatex}

\renewcommand{\theequation}{\thesection.\arabic{equation}}

\newcommand{\A}{{\mathcal A}}
\newcommand{\T}{{\mathcal T}}
\newcommand{\Pt}{{\mathcal P}}
\begin{document}

\begin{frontmatter}

%Title, authors and addresses

% use the thanksref command within \title, \author or \address for footnotes;
% use the corauthref command within \author for corresponding author footnotes;
% use the ead command for the email address,
% and the form \ead[url] for the home page:
% \title{Title\thanksref{label1}}
% \thanks[label1]{}
% \author{Name\corauthref{cor1}\thanksref{label2}}
% \ead{email address}
% \ead[url]{home page}
% \thanks[label2]{}
% \corauth[cor1]{}
% \address{Address\thanksref{label3}}
% \thanks[label3]{}

\title{An $H_1$-BMO duality theory in the framework of semigroups of operators}

% use optional labels to link authors explicitly to addresses:
% \author[label1,label2]{}
% \address[label1]{}
% \address[label2]{}

\author{Tao Mei} \thanks{The author was supported in part by NSF grant DMS 0901009.}
\address{Dept. of Math., Wayne State Univ., Detroit, MI.}

\begin{abstract}
Let $(M,\mu)$ be a sigma-finite measure space. Let $(T_t)$ be a semigroup of positive preserving maps on $(M,\mu)$ with standard assumptions. We prove a $H_1$-BMO duality theory with assumptions only $(T_t)$ itself. The BMO is defined as spaces of functions $f$ such that $\sup_t\|T_t|f-T_tf|^2\|<\infty$. The $H_1$ is defined by square functions of P. A. Meyer's gradient form. Our argument does not rely on the geometric/metric structure of $M$ nor on the kernel of the semigroups of operators. This allows our main results extend to the noncommutative setting as well, e.g. the case where $L_\infty(M,\mu)$ is replaced by von Neumann algebras with a semifinite trace. We also prove a Carleson embedding theorem for semigroups of operators.
\end{abstract}


\end{frontmatter}

%\setlength{\oddsidemargin}{0in} \setlength{\evensidemargin}{0in}
%\setlength{\textwidth}{15.5cm} \setlength{\textheight}{22cm}
%\setlength{\topmargin}{-1cm}
%\begin{document}

%\begin{center}
%\mathstrut

%{\Large {\bf Tent Spaces Associated with Semigroup of operators}}

%\bigskip

%{\large Tao Mei}\footnote{
%The author was supported in part by a Young Investigator Award of the N.S.F supported summer workshop in Texas A\&M university 2007.
%}
%\end{center}

%\begin{quotation}
%{\bf Abstract.} We study Tent spaces on general measure space
%$(\Omega, \mu)$. We assume there exists a semigroup of positive
%operators on $L^p(\Omega, \mu)$ satisfying a monotone property but do not assume any
%geometric/metric structure existing on $\Omega$. The semigroup plays
%the role of integrals on cones and cubes of Euclidean space.
%We then study BMO spaces on general measure spaces and get
%an analogue of Fefferman's $H_1$-BMO duality theory. We also get a $H_1$-BMO duality inequality without assuming
%the monotone property.

%All the results are proved in a more general setting, that is for elements of a
%semifinite von Neumann algebra.
%\end{quotation}
\setcounter{section}{-1}

\section{Introduction}

\setcounter{theorem}{0}\setcounter{equation}{0} E. Stein ([St70]) studied a ``universal" $H^p$ theory for $1<p<\infty$ in the frame work of semigroup of operators. After Stein's work, many other mathematicians (e.g. M. Cowling, P. A. Meyer, N. Varopoulos, Doung/Yan, Auscher/McIntosh and their coauthors) have been working on Fourier multipliers and Hardy/BMO spaces associated with semigroups of operators (see [Cow83], [FS82], [Mey74], [Var80], [DY05], [ADM04], [HM09], [HLMMY] etc.). In particular, Doung/Yan proved an $H_1$-BMO duality for semigroups of operators with heat kernel bounds in their remarkable article [DY05]. S. Hofmann and S. Mayboroda proved an $H_1$-BMO duality for semigroups of operators generated by divergence form elliptic operators. A novelty of Doung/Yan's definitions of BMO and $H_1$ is that they are according to the growth of the kernel of the underlying semigroups of operators.

One aim of those work is to establish a Hardy spaces theory which relies on less geometric/metric properties of Euclidean spaces.
This article continues the effort to this aim and provides an abstract approach to an $H_1$-BMO duality with assumptions only on the semigroups of operators. By ``abstract", we mean that the formulation of the duality is unified and the corresponding constants are absolute, as in Stein's work for Littlewood-Paley theory in the case of  $1<p<\infty$. %In this sense, Doung/Yan's remarkable work ([DY05]) does not really fit our desire.

%A main difficulty for such a universal theory is the missing of nice geometric property of Euclidean spaces in an abstract setting. The semigroups of operators provide an appropriate  framework to overcome this difficulty.

Let $(M,\mu)$ be a sigma-finite measure space. Let $(T_t)$ be a standard semigroup of operators on $L^p(M)$. Under standard assumptions (see Def \ref{standard}), $T_t$ can be viewed as a natural alternative to the classical mean value operator. This led to two natural definition of BMO-norms,
\begin{eqnarray*}
\|f\|_{BMO(\T)}&=&\sup_t\|T_t|f-T_tf|^2\|_\infty^\frac12.\\
\|f\|_{bmo(\T)}&=&\sup_t\|T_t|f|^2-|T_tf|^2\|_\infty^\frac12.
\end{eqnarray*}

These two BMO norms are equivalent to the usual interval BMO norms if $(T_t)$ are the heat semigroups on ${\Bbb R}^n$. But $\|\cdot\|_{BMO(\T)}$ and $\|\cdot\|_{bmo(\T)}$ may not be equivalent to each other in general. Let BMO$(\T)$ and bmo$(\T)$ be Banach spaces characterized by the corresponding BMO norms. In [JM12], Junge and the author of this paper prove that an interpolation result holds between $BMO(\T)$ and $L_1(M)$ and,  the interpolation theory for bmo$(\T)$ and $L_1(M)$ also hold with an additional continuity assumption (see Lemma \ref{JM1}).

Let us consider the following analogues of  Lusin area integral and Littlewood-Paley G function.
\begin{eqnarray*}
S_\Gamma(f)&=&(\int_0^\infty T_s\Gamma(T_sf)ds)^\frac12.\\
G_\Gamma(f)&=&(\int_0^\infty \Gamma(T_sf)ds)^\frac12.
\end{eqnarray*}
Here $\Gamma(\cdot)$ is P. A. Meyer's gradient form (Carr\'e du Champ) for semigroups of operators (see Def \ref{Gamma}). If $(T_t)$ are the classical heat semigroups (or Ornstein-Uhlenbeck) semigroups on ${\Bbb R}^n$ (with the Gaussian measure), then
$$\Gamma(f) = |\partial_x f|^2.$$
This is also the case, if $T_t$ is the Laplace-Beltrami operator on a Riemannian manifold. We use $\int_0^\infty T_s(\cdot)ds$
in the definition of the S-function as an alternative to the integration on the
cones.

Let us define $H_1^S(\T)$ (resp. $H_1^G(\T)$) as the space of all $f\in L_1(M)$ such that $\|f\|_{H_1^S(\T)}=\|S_\Gamma(f)\|_1<\infty$ (resp. $\|f\|_{H_1^G(\T)}=\|G_\Gamma(f)\|_1<\infty$).

Denote by $\tau(\cdot)$ the integration operator $\int\cdot d\mu$ on $M$.  Our first result is the following duality inequality between $H_1$ and BMO.

\begin{theorem}
Let $(T_t)$ be a standard semigroup of operators satisfying the $\Gamma_2\geq0$ condition (\ref{Gamma_2}). Then bmo$(\T)\subset (H_1^S(\T))^*$ and
\begin{eqnarray*}
\tau (fg)\leq c_1\|f\|^{\frac12}_{H_1^S(\T)}\|f\|^{\frac12}_{H_1^G(\T)}\|g\|_{bmo(\T)}\leq c_2\|f\|_{H_1^S(\T)}\|g\|_{bmo(\T)},
\end{eqnarray*}
with absolute constants $c_1,c_2$.\end{theorem}

Our attention then turns to the following questions.
\begin{itemize}
\item
When does the other direction $bmo(\T) \supseteq (H_1^S(\T))^*$ hold?

\item
When does $bmo(\T)=BMO(\T)$?

\item When does $H_1^S(\T)=H_1^G(\T)$?
\end{itemize}

The answers to them  are all ``yes" if  the semigroup $(T_t)$  satisfies two more conditions.
\begin{theorem}\label{end}
Let $T_t$ be a standard semigroup of operators satisfying the $\Gamma_2\geq0$ condition (\ref{Gamma_2}). Assume, in addition, that there exist constants $c_3,c_4, r>0$ such that

(i)  $\|(T_{t+\varepsilon t}-T_{t})f\|_1\leq c_3 \varepsilon^r\|f\|_1,$ for all $\varepsilon>0, t>0$ and $f\in L_1(M)$.

(ii) Denote $M_t=\frac1t\int_{0}^{t}T_sds$.
\begin{eqnarray}\label{Lhalf1}
\|(M_{8t}|T_tf|^2)^\frac12\|_{L_1(M)}\leq c_4\|f\|_{L_1(M)},
\end{eqnarray}
for all $t>0$ and $f\in L_1^+(M)$.

 Then every linear functional $\ell$ on $H_1^S(\T)$ can be represented as $\tau (g\cdot)$ with $\|g\|_{bmo(\T)}\simeq c\|\ell\|_{(H_1^S(\T))^*}$. Moreover,

$$bmo(\T)=BMO(\T)\ \  {\rm and} \ \ H_1^G(\T)=H_1^S(\T).$$
The equivalent constants only depend on $c_3,c_4, r$.
\end{theorem}
 %Examples of semigroups of operators satisfying these conditions are given in Section 2.

Our argument does not rely on any geometric/metric structure of $M$ nor on the kernel of the semigroups of operators. In fact, our argument only needs an abstract $L_\infty$ space and a standard semigroup of operators on this space satisfying the assumptions in Theorem 0.1 and 0.2. We consider these assumptions as an reflection of the geometric properties of $M$, as what in the mind of many other researchers, e.g. D. Bakry, etc. This abstract argument allows to extend our main results to the noncommutative setting, which is our final goal. The drawback is that we require the semigroups of operators are positive preserving, while Doung/Yan and Hofmann/Mayboroda's theorems go beyond those type of semigroups of operators.

 From the point of view of functional analysis, every $L_\infty$ space on a sigma-finite measure spaces is a commutative von Neumann algebra. This led to defining noncommutative $L^p$ spaces as von Neumann algebras ${\cal M}$ with ``nice" linear functionals, called traces and denoted by $\tau$, which play the role of integration with respect to $\mu$.

 The importance of analyzing semigroups of operators on von Neumann algebras has
been impressively demonstrated by the recent work of Popa and
Ozawa [OP10] and also occurs in the work of Shlyahktenko/Connes [CS05] on Betti numbers for
von Neumann algebras. M. Junge and the author of this article build up a connection between semigroups of operators on von Neumann algebras and M. Rieffel's quantum metric spaces (see [Rie], [JM10]).

Noncommutative analogues of analytic Hardy spaces have been developed mainly by W. Arveson (see [A67]).  Pisier/Xu and their collaborators have established the noncommutative theory of martingale Hardy spaces (see [PX97]). Junge-Le Merdy-Xu studied noncommutative real Hardy spaces for $1<p<\infty$ in [JLX06]. In particular, an $H_1$-BMO duality for noncommutative martingales is proved in [PX97] and [JX03].  [M07] proves an analogue of the classical real variable $H_1$-BMO duality in the semi-commutative case. [M08] is a first try on a real variable $H_1$-BMO duality in the general noncommutative setting. [JM12] established an interpolation result between semigroup BMO spaces and noncommutative $L_p$ spaces.  Noncommutative fourier multiplier theories are further developed in [JM10] and [JMP] by using BMO spaces associated with semigroups of operators defined above. This article improves the method developed in [M08] and extends Theorem 0.1 and 0.2 to the noncommutative case in Section 3.

\section{Preliminaries}

\subsection{Semigroups of operators}

Let $(M ,\sigma , \mu )$ be a sigma-finite measure space. Let $L^p(M)$ be the space of all complex valued $p$-integrable functions on $M$. Denote by $f^*$  the pointwise complex conjugate of a function $f$ on $M$.

\begin{definition}\label{standard}
A family of operators  $(T_y)_y$ is a {\it standard} semigroup of operators, if
$T_{y_1}T_{y_2}=T_{y_1+y_2},T_0=id$ and

(i) $T_y$ are contractions on $L^p(M)$ for all $1\leq p\leq\infty.$

(ii) $T_y$ are symmetric, i.e. $T_y=T^*_y$ on $L^2(M)$.

(iii) $T_y(1)=1$

(iv) $T_y(f)\rightarrow f$ in $L^2$ as $y\rightarrow0+$ for $f\in L^2.$
\end{definition}
The conditions (i), (iii) above imply $T_y$ is \textit{positivity preserving}
for each $y$, i.e. $T_y(f)\geq 0$ if $f\geq 0.$ We will need the following Kadison-Schwarz inequality for unital
(completely) positive contraction $T$ on $L_p(M)$,
\begin{eqnarray}\label{cp}
|T(f)|^2\leq T(|f|^2),\ \ \ \ \forall f\in L_p(M).
\end{eqnarray}

A standard semigroup $(T_y)$ always admits an
infinitesimal generator $L=\lim_{y\rightarrow 0}\frac {T_y-id}{y}.$
$L$ is a unbounded operator densely defined on $L_2(M)$. We will write $T_y=e^{yL}$.
Some of the conditions (i)-(iv) may be weaken but that is beyond the main interests of this article.

\begin{definition}
P. A Meyer's gradient form $\Gamma$ (also called ``Carr\'e du Champ") associated with $T_t$ is defined as,
 \begin{eqnarray}\label{Gamma}
  2\Gamma(f,g) = L(f^{*}g)-(L(f^{*})g)-f^{*}(L(g)),
  \end{eqnarray}
for $f,g$ with $f^*,g, f^*g\in D(L)$. 
\end{definition}
When $f=g$, we simply write $\Gamma(f)=\Gamma(f,g)$.

For convenience, we assume
that there exists a $^*$-algebra $\A$ which is weak$^*$ dense in
$L_\infty(M)$ such that $T_s(\A)\subset\A\subset D(L)$. This assumption is
to guarantee that $\Gamma(T_sf,T_sg)$ make senses for $f,g\in \A$,
which is not easy to verify in general, although the other form
$T_t\Gamma(T_sf,T_sg)$ is what we need essentially in this article
and can be read as
$LT_t(T_sf^{*}T_sg)-T_t((LT_sf^{*})T_sg)-T_t(T_sf^{*}(LT_sg))$ for
any $f,g\in L_p(M),1\leq p\leq\infty, s,t>0$.

It is easy to verify that for $L=\triangle=\frac{\partial ^2}{\partial^2 x}$, $\Gamma(f,g)=\frac{\partial f^*}{\partial x}\cdot \frac{\partial g}{\partial x}$.
It is well known that the positive-preserving property of a standard semigroup of operators implies that $\Gamma(f)\geq0$ for all $f$.
\begin{definition}
Bakry-\'Emery's iterated gradient form $\Gamma_2$ is defined as
 \begin{eqnarray}
\Gamma_2(f,g) = L\Gamma(f,g)-\Gamma(f^{*}L(g))-\Gamma((Lf^{*})g),
  \end{eqnarray}
for $f,g\in \A$.
\end{definition}
When $f=g$, we simply write $\Gamma_2(f)=\Gamma_2(f,g)$.
For $L=\triangle=\frac{\partial ^2}{\partial^2 x}$, $\Gamma_2(f,g)=\frac{\partial^2 f^*}{\partial x^2}\cdot \frac {\partial^2 g}{\partial x^2}$.
However, $\Gamma_2(f,f)\geq0$ is not always true.
For $L$ being the Laplace-Beltrami operator on a complete manifold, $\Gamma_2(f)\geq0$ is equivalent to the positivity of the Ricci curvature of the manifold.
See [BBG12] and references therein for more details on $\Gamma_2$ and examples of semigroups of operators satisfying the ``$\Gamma_2\geq0$" condition, which is the so-called curvature-dimension criterion  $CD(0,\infty)$ in [BBG12]. We still use the relative old notation ``$\Gamma_2\geq0$" because it makes sense even in the noncommutative setting.

It is easy to check that, for a standard semigroup of operator $(T_t)_t$, $\Gamma_2(f)\geq0$ iff
 \begin{eqnarray}\label{Gamma_2}
  \Gamma (T_v f) \leq T_v \Gamma (f)
  \end{eqnarray}
for all $v>0,f\in \A$.

We will need the following Lemma due to P.A. Meyer. We add a short proof for the convenience of the reader.
\begin{lemma}\label{lemma}
 For any $f\in L_p(M),1\leq p\leq\infty, s>0$, we have
 \[ T_s|f|^2-|T_sf|^2 = 2 \int_{0}^s  T_{s-t}\Gamma(T_tf) dt  .\]
\end{lemma}

\begin{proof} For $s$ fixed, let
\[
F_t=T_{s-t}(|T_tf|^2).
\]
Then
\begin{eqnarray*}
\frac{\partial T_{s-t}(|T_tf|^2)}{\partial t} &=&\frac{\partial T_{s-t}}{%
\partial t}(|T_tf|^2)+T_{s-t}[(\frac{\partial T_t}{\partial t}%
f^{*})f]+T_{s-t}[f^*(\frac{\partial T_t}{\partial t}f)] \\
&=&-T_{s-t}\Gamma (T_tf).
\end{eqnarray*}
Therefore
\begin{eqnarray*}
T_s|f|^2-|T_sf|^2 =-F_s+F_0=\int_0^sT_{s-t}\Gamma (T_tf)dt.
\end{eqnarray*}
\end{proof}
\begin{definition}
\label{sub} Given a standard semigroup of operators $(T_y)_y$ with an infinitesimal generator
$L$, the semigroup $(P_y)_y$ defined as
\[
P_y=e^{-y\sqrt{-L}}
\]
is again a standard semigroup of operators. We call it the subordinated
Poisson semigroup of $(T_y)_y$.
\end{definition}

Note $P_y$ is chosen such that
\begin{eqnarray}
(\frac{\partial ^2}{\partial s^2}+L)P_s=0.  \label{1}
\end{eqnarray}
It is well known that (see [St2])
\begin{eqnarray}
P_y=\frac 1{2\sqrt{\pi }}\int_0^\infty ye^{-\frac{y^2}{4u}}u^{-\frac
32}T_udu.  \label{idpy}
\end{eqnarray}
Apply $\Gamma(f)\geq0$ and (\ref{idpy}), it is easy to deduce that $\Gamma_2\geq 0$ also
implies $\Gamma(P_vf,P_vf)\leq P_v\Gamma(f)$ for any $v>0$.

(\ref{idpy}) also implies that
\begin{eqnarray}
\frac{P_y}y(f)\leq \frac{P_t}t(f) \label{sbd} \ \ {\rm and}\ \
|(P_y-P_{y+t})f|\leq \frac {8t}y P_{\frac y2}f,
\end{eqnarray}
for any $0\leq t\leq y, f\geq0$,
since $T_u$ is positive and $e^{-\frac{y^2}{4u}}u^{-\frac 32}$ is a function
decreasing with respect to $y.$ \smallskip

The classical heat semigroup and Ornstein-Uhlenbeck semigroup on $\Bbb{R}^n$ are typical examples of
standard semigroups of operators. They can be presented as
\begin{eqnarray}
T_t&=&e^{t\triangle}\\
O_t&=&e^{t(\frac12\triangle+x\cdot \frac{\partial }{\partial x})}
\end{eqnarray}
with $\triangle= \frac{\partial^2 }{\partial x^2} =\sum_{i=1}^n\frac{\partial ^2}{\partial x_i^2}, $ the
Laplacian operator on $\Bbb{R}^n$.

It is easy to check that for these two semigroups,

\begin{eqnarray}\Gamma(f,g)=\frac{\partial }{\partial x}f^*\cdot\frac{\partial }{\partial x}g,\end{eqnarray}

and they both satisfy the $\Gamma_2\geq0$ condition (\ref{Gamma_2}).
\subsection{BMO spaces associated with semigroups of operators}
Recall we set
\begin{eqnarray}
\|f\|_{\mathrm{bmo}({\cal T}) }&=&\sup_{0<t<\infty}\| T_t|f|^2-|T_tf|^2\|_{L_\infty}^\frac12.\\
\|f\|_{\mathrm{BMO}({\cal T}) }&=&\sup_{0<t<\infty}\| T_t|f-T_tf|^2\|_{L_\infty}^\frac12. \label{BMOT}
\end{eqnarray}

It is easy to see by Kadison-Schwarz inequality (\ref{cp}) that for any positive sequence $t_n$ which converges to $0$,

\begin{eqnarray}
\|f\|_{\mathrm{bmo}({\cal T}) }&\simeq&\sup_{n\in {\Bbb N}}\| T_{t_n}|f|^2-|T_{t_n}f|^2\|_{L_\infty}^\frac12.\\
\|f\|_{\mathrm{BMO}({\cal T}) }&\simeq&\sup_{n\in {\Bbb N}} \|T_{t_n}|f-T_{t_n}f|^2\|_{L_\infty}^\frac12. \label{disBMOT}
\end{eqnarray}

\begin{lemma} \label{JM2}([JM12]) Let $(T_t)$ be a standard semigroup of operators.  Then
(i) $\|f\|_{bmo(\T)}=0$ iff $\|f\|_{BMO(\T)}=0$ iff $f\in ker(L)=\{f\in D(L) Lf=0\}$.
(ii) If in addition  $(T_t)$ satisfies the $\Gamma_2\geq0$ condition (\ref{Gamma_2}), then
 $$\|f\|_{BMO(\T)}\simeq \|f\|_{bmo(\T)}+\sup_t
 \|T_tf-T_{2t}f\|.$$
\end{lemma}

\begin{lemma}\label{JN}Suppose $(T_t)$ is a standard semigroup of operators. Then
\begin{eqnarray}\label{JNs}
\sup_t\|(T_{t}|f-T_{t}f(\cdot)|^p)(\cdot)\|_\infty^\frac1p\simeq^p \|f\|_{bmo(\T)},
\end{eqnarray}
for any $0<p<\infty$. And, if $(T_t)$ satisfies the $\Gamma_2\geq$ condition (\ref{Gamma_2}), then
\begin{eqnarray}\label{JNSS}
\sup_t\|T_{t}|f-T_{t}f|^p\|_\infty^\frac1p\simeq^p\|f|\|_{BMO(\T)},
\end{eqnarray}
for any $0< p<\infty$.
\end{lemma}
\begin{proof}
Let $W_t$ be a Markov process, so that
\begin{eqnarray}
E_sf(W_t)=T_{t-s}f(W_s),
\end{eqnarray}
for any $0\leq s<t<\infty$. The Markov process $(W_t)$ can be constructed  by setting its covariance function $C_{s,t}=K_{t-s}$, the kernel of $(T_t)$.
Fix a $t>0$, we consider the martingale
$E_s(f(W_t)), s=2^{-n}t, n\in{\Bbb N}$.
By the John-Nirenberg inequality for the little BMO space of martingales, we have,
\begin{eqnarray}\label{JNM}
\sup_{s=2^{-n}t}\|E_s(|f(W_t)-E_sf(W_t)|^p)\|_\infty\simeq^p\sup_{s=2^{-n}t}\|E_s|f(W_t)-E_sf(W_t)|\|_\infty^p
\end{eqnarray}
On the other hand, by the Markov property,  we have that, on the set $\{W_s=v\}$,
\begin{eqnarray}
E_s(|f(W_t)-E_sf(W_t)|^p)&=&E_s(|f(W_t)-T_{t-s}f(v)|^p)\nonumber
\\&=&(T_{t-s}|f-T_{t-s}f(v)|^p)(v).
\end{eqnarray}
Therefore,
\begin{eqnarray}\label{idSM}
\|E_s(|f(W_t)-E_sf(W_t)|^p)\|_\infty=\|(T_{t-s}|f-T_{t-s}f(\cdot)|^p)(\cdot)\|_\infty.
\end{eqnarray}
Combining (\ref{JNM}) and (\ref{idSM}) we obtain the John-Nirenberg inequality for $bmo(\T)$,
\begin{eqnarray}
\sup_t\|(T_{t}|f-T_{t}f(\cdot)|^p)(\cdot)\|_\infty^\frac1p\simeq^p \|f\|_{bmo(\T)},
\end{eqnarray}
for any $0<p<\infty$.

Now, for any $p>2$,
\begin{eqnarray*}
&&(T_{t}|f-T_{t}f|^p(\cdot))^\frac1p\\
&\leq& (T_{t}|f-T_{t}f(\cdot)|^p(\cdot))^\frac1p+(T_{t}|T_tf-T_{t}f(\cdot)|^p(\cdot))^\frac1p\\
&\leq& [T_{t}(|f-T_{t}f(\cdot)|^p)(\cdot)]^\frac1p+[T_{2t}|f-T_{2t}f(\cdot)|^p(\cdot)]^\frac1p+(|T_t(f)-T_{2t}f|)(\cdot).
\end{eqnarray*}

 Taking supremum on both sides, we get by (\ref{JNs}) and Lemma \ref{JM2} (ii) that
\begin{eqnarray}
\sup_t\|T_{t}(|f-T_{t}f|^p)\|^\frac1p_\infty&\leq&
 2c_p\|f\|_{bmo(\T)}+\sup_t\|T_t(f)-T_{2t}f\|_\infty\nonumber\\
&\leq&cc_p\|f\|_{BMO(\T)}.\label{p>2}
\end{eqnarray}
For $0<q<2$, let $p=4-q$, by H\"older's inequality,
\begin{eqnarray*}
T_{t}(|f-T_{t}f|^2)(\cdot)&=& T_{t}(|f-T_{t}f|^\frac q2|f-T_{t}f|^{\frac {4-q}2})(\cdot)\\
&\leq& [T_{t}(|f-T_{t}f|^q)(\cdot)]^\frac12\cdot[T_{t}(|f-T_{t}f|^p)(\cdot)]^\frac12.
\end{eqnarray*}
Taking supremum on both sides and applying (\ref{p>2}), we get
\begin{eqnarray*}
\|f\|_{BMO(\T)}\leq c_q\sup_t\|T_{t}(|f-T_{t}f|^q)\|_\infty^\frac14\|f\|_{BMO(\T)}^\frac p4.
\end{eqnarray*}
Therefore,
\begin{eqnarray*}
\|f\|_{BMO(\T)}\leq c_q\sup_t\|T_{t}(|f-T_{t}f|^q)\|_\infty^\frac1q.
\end{eqnarray*}
\end{proof}

\begin{lemma} \label {JM1}([JM12]) Assume that $(T_t)$ is a standard semigroup of operators. Then
$$ [BMO(\T),L^0_1(M)]_{\frac{1}{p}} = L^0_p(M) $$
for
 $1<p<\infty$. If, in addition, $(T_t)$ admits a Markov dilation which has a. u. continuous path, then
$$ [bmo(\T),L^0_1(M)]_{\frac{1}{p}} = L^0_p(M) $$
for  $1<p<\infty$.
\end{lemma}
Here $L^0_p(M)=L_p(M)/kerL$.

\begin{remark}
Lemma \ref{JM2} could include here a proof of Lemma \ref{JM1} in a few lines, if the little BMO space of martingales works as an interpolation-end point for $L_p$ martingale spaces. But that is not the case in general.
\end{remark}

\section{Proof of Theorem 0.1, 0.2}

\setcounter{theorem}{0}\setcounter{equation}{0}

Recall, for $f\in L_1(M)$, we set
\begin{eqnarray*}
S_\Gamma(f)&=&(\int_0^\infty T_s\Gamma(T_sf)ds)^{\frac12}, \\
G_\Gamma(f)&=&(\int_0^\infty \Gamma(T_sf)ds)^{\frac12}.
\end{eqnarray*}

Set
\begin{eqnarray*}
\|f\|_{\mathcal{H}_{1}^{S}({\cal T})} &=&\|S_\Gamma(f)\|_{L_1}, \\
\|f\|_{\mathcal{H}_{1}^{G}({\cal T})}&=&\|G_\Gamma(f)\|_{L_1}.
\end{eqnarray*}

It is easy to see that
\begin{eqnarray}
\|f\|_{\mathcal{H}_{1}^{G}}\leq 2\|f\|_{\mathcal{H}_{1}^{S}},
\label{HGS}
\end{eqnarray}
by $\Gamma_2\geq0$.
Let $\mathcal{H}^S_{1},\mathcal{H}^G_{1},\mathrm{bmo}({\T})$ and BMO$_({\T})$ be the corresponding Banach spaces.
%We define the corresponding row spaces and column-row spaces in the usual
%way.

%\begin{theorem}
%\label{end1.1} Suppose $(T_y)_{y\geq 0}$ is a standard semigroup of operators satisfying $\Gamma_2\geq0$. We have
%\[
%\mathrm{BMO}({\cal T})\subset (H_{1}^S)^{*}.
%\]
%And
%\[
%|\langle f,\varphi \rangle|^2\leq c\|S_\Gamma(f)\|_1\|G_\Gamma(f)\|_1\|\varphi\|^2_{{\rm BMO}({\cal T})}.
%\]
%\end{theorem}
 Set
truncated square functions $S_{s},G_s $ as follows:
\begin{eqnarray}
S_{s} &=&(\int_s^\infty T_{y-\frac s2}(\Gamma( T_{y+\frac s2}f)dy)^{\frac 12}  \label{Ss} \\
{G}_s &=&(\int_s^\infty \Gamma( T_{2y}f)dy)^{\frac
12}.  \label{Gs}
\end{eqnarray}
$S_{s},G_s$ are constructed to satisfy our key Lemma.

\begin{lemma}
\label{endlem1.1}
\begin{eqnarray}
{G}_s&\leq& S_{s};  \label{endlemma1} \\
\frac{dT_{(a+b)s}(S_{s})}{ds}&\geq& \frac{a+b}b T_{{a s}}(\frac{dT_{bs}(S_{s})}{ds}),\label{endlemma2}\\
\frac{dT_{\frac s2}(S_{s})}{ds}&\leq& 0,  \label{endlemma}
\end{eqnarray}
for any $a,b>0$.
\end{lemma}

\textbf{Proof}. (\ref{endlemma1}) is true because of the fact
\begin{eqnarray*}
\Gamma (T_{2y}f) \leq T_{y-\frac s2}\Gamma(
T_{y+\frac s2}f),
\end{eqnarray*}
which follows from the $\Gamma_2\geq0$ condition (\ref{Gamma_2}).

Apply (\ref{Gamma_2}) again, we get $S_{s}\geq S_{t}$ for any $s\leq t$, then
\begin{eqnarray*}
&&T_{(a+b)s+(a+b)\Delta s}(S_{s+\Delta s})-T_{(a+b)s}(S_{s})\\ &=&T_{{a s}}[T_{bs+(a+b)\Delta s}(S_{s+\Delta s})-T_{bs}(S_{s})] \\
&\geq&T_{a s}[T_{bs+(a+b)\Delta s}(S_{s+\frac{a+b}b\Delta s})-T_{bs}(S_{s})].
\end{eqnarray*}
Divide by $\Delta s$ both sides, we get (\ref
{endlemma2}).

We go to prove (\ref{endlemma}).
By (\ref{cp}) and $\Gamma_2\geq0$ condition (\ref{Gamma_2}), we get
\begin{eqnarray*}
T_{\Delta s}S_{s+2\Delta s}
&=&T_{\Delta s}(\int_{s+2\Delta s}^\infty T_{y-\frac s2-\Delta
s}\Gamma(T_{y+\frac s2+\Delta s}f)dy)^{\frac
12} \\
&\leq&(\int_{s+2\Delta s}^\infty T_{y-\frac s2}\Gamma(T_{y+\frac s2+\Delta s}f)dy)^{\frac 12}\\
&\leq&(\int_{s+2\Delta s}^\infty T_{y-\frac
s2+\frac{\Delta s}2}\Gamma(T_{y+\frac s2+\frac{\Delta s}2}f)dy)^{\frac 12}\\
&=&(\int_{s+2\Delta s+\frac{\Delta s}2}^\infty T_{u-\frac
s2}\Gamma(T_{u+\frac s2}f)du)^{\frac 12} \\
&\leq&(\int_s^\infty T_{y-\frac s2}\Gamma( T_{y+\frac
s2}f)dy)^{\frac 12}=S_{s}.
\end{eqnarray*}
Then
\begin{eqnarray*}
T_{\frac{s+2\Delta s}2}S_{s+2\Delta s}-T_{\frac s2}S_s\leq 0.
\end{eqnarray*}
Taking $\Delta s\rightarrow 0$ we obtain (\ref
{endlemma}).

\begin{lemma}
\label{fphis} We have
\begin{eqnarray*}
&&|\tau \int_0^\infty \Gamma(T_{2s}f,T_{(2+v)s}\varphi_s) ds|\\&\leq& (8+4v)^{\frac12}\|G_\Gamma(f)\|_1^{\frac 12} (-\tau \int_0^\infty (T_{(\frac32+v)y}\int_0^y\Gamma(\varphi_s)ds) \frac{ \partial
T_{\frac y2}(S_{y})}{\partial y} dy)^{\frac 12}
\end{eqnarray*}
for any $v>0$ and $f\in \A$ and any  $\varphi _s\in\A$.
\end{lemma}

\textbf{Proof.} We can assume $G_s,S_{v,s}$ are invertible by approximation. By (\ref{cp}%
), (\ref{endlem1.1}) and Cauchy-Schwarz inequality, we get
\begin{eqnarray*}&&|\tau \int_0^\infty \Gamma(T_{2s}f, T_{(2+v)s}\varphi_s)ds|\\ &\leq&(\tau\int_0^\infty \Gamma(T_{2s}f){G}_s^{-1}ds )^{\frac 12} (\tau\int_0^\infty \Gamma( T_{(2+v)s}\varphi_s){G_s} ds)^{\frac 12}\\
&\stackrel{def}{=}&I^{\frac 12}II^{\frac 12}.
\end{eqnarray*}

Note that $-\frac{\partial {G}_s^2}{\partial s}=\Gamma(T_{2s}f).$ For $I,$ we have
\begin{eqnarray*}
I =\tau \int_0^\infty -\frac{\partial {G}_s^2}{\partial s}{G}_s^{-1} ds =2{
\tau}\int_0^\infty -\frac{\partial {G}_s}{\partial s}ds =2\|G_0\|_1.
\end{eqnarray*}
We estimate $II$.   By (\ref{endlemma1})and $\Gamma_2\geq0$ we have
\begin{eqnarray}
II&\leq&{\tau}\int_0^\infty \Gamma(T_{(2+v)s}\varphi_s)S_sds\\
&\leq&{\tau}\int_0^\infty \Gamma(\varphi_s)T_{(2+v)s}S_sds
  \nonumber \\
&=&{\tau}\int_0^\infty \Gamma(\varphi_s)\int_s^\infty-\frac{\partial T_{(2+v)y}(S_y)%
}{\partial y}dyds  \nonumber \\
&=&-\tau\int_0^\infty \int_0^y\Gamma(\varphi_s)ds\frac{\partial T_{(2+v)y}(S_y)}{%
\partial y}dy.  \label{II}
\end{eqnarray}
Applying (\ref{endlemma}), with $a,b=\frac32+v,\frac12$ to (\ref{II}), we get
\begin{eqnarray*}
II&\leq&-(4+2v){\tau}\int_0^\infty \int_0^y\Gamma(\varphi_s)dsT_{(\frac32+v)y}(\frac{
\partial T_{\frac y2}(S_y)}{\partial y})dy \\
&=&-(4+2v){\tau}\int_0^\infty (T_{(\frac32+v)y}\int_0^y\Gamma(\varphi_s)ds)\frac{ \partial
T_{\frac{y}2}(S_y)}{\partial y}dy.
\end{eqnarray*}
Combining the estimates of I and II, we get the desired inequality.

\medskip \textbf{Proof of Theorem 0.1.} Note $\tau Lf=0, \tau (Lf)g=\tau f(Lg)$ for $f,g\in \A$. By the definition of $\Gamma$ (\ref{Gamma}),
\begin{eqnarray*}
\tau f\varphi^*=-\tau\int_0^\infty \frac {\partial}{\partial s} (T_{2s}fT_{(3+v)s}\varphi^*)ds
=(5+v)\tau\int_0^\infty \Gamma( T_{2s}f,T_{(3+v)s}\varphi)ds
\end{eqnarray*}

Applying Lemma \ref{fphis} to $\varphi_s=T_s\varphi$, we get
\begin{eqnarray*}
|\tau f\varphi^*|\leq  12\sqrt3 \|G_\Gamma(f)\|_1^{\frac 12} (-\tau \int_0^\infty \int_0^yT_{(\frac32+v)y}\Gamma(T_s\varphi)ds \frac{ \partial
T_{\frac y2}(S_{y})}{\partial y} dy)^{\frac 12},
\end{eqnarray*}
for any $0<v<1$.

Denote $M_y=\frac1y\int_0^y T_sds$. Integrate on $v$ for $0<v<1$, we have
\begin{eqnarray*}
|\tau f\varphi^*|^2\leq  c \|G_\Gamma(f)\|_1 (-\tau \int_0^\infty \int_0^yM_yT_{(\frac32)y}\Gamma(T_s\varphi)ds \frac{ \partial
T_{\frac y2}(S_{y})}{\partial y} dy),
\end{eqnarray*}
Then, by (\ref{endlemma}), we have
\begin{eqnarray*}
|\tau f\varphi^*|^2&\leq& c \|G_\Gamma(f)\|_1 (\sup_y\|\int_0^yM_yT_{(\frac32)y}\Gamma(T_s\varphi)ds\|_\infty) \int_0^\infty-\frac{ \partial
T_{\frac y2}(S_{y})}{\partial y} dy\\
&\leq&  c \|G_\Gamma(f)\|_1 (\sup_y\|\int_0^yM_yT_{(\frac32)y}\Gamma(T_s\varphi)ds\|_\infty) \|S_{\Gamma}(f))\|_{1}.
\end{eqnarray*}

On the other hand, note $M_yT_{\frac y2+s}\leq 3M_{3y}$ for $0<s<y$, we have
\begin{eqnarray*}
\int_0^yM_yT_{(\frac32)y}\Gamma(T_s\varphi,T_s\varphi)ds&=&\int_0^yM_yT_{\frac y2+s}T_{y-s}\Gamma(T_s\varphi)ds\\
&\leq& 3M_{3y}\int_0^yT_{y-s}\Gamma(T_s\varphi)ds\\
(\rm Lemma \ \ref{lemma})&=&3M_{3y}(T_y|\varphi|^2-|T_y \varphi|^2).
\end{eqnarray*}
Therefore,
$$\sup_y\|\int_0^yM_yT_{(\frac32)y}\Gamma(T_s\varphi)ds\|_\infty\leq 3\|\varphi\|_{{\rm bmo}({\cal T})}^2,$$
because $M_{3y}$ is a bounded operator on $L_\infty$.

We concluded that
\begin{eqnarray*}
|\tau (f\varphi^*)|^2\leq  c \|G_\Gamma(f)\|_1\|S_\Gamma(f)\|_1\|\varphi\|_{\rm bmo ({\cal T})}^2.
\end{eqnarray*}
\qed

\begin{remark}
Known examples of semigroups operators satisfying the $\Gamma_2\geq0$ conditions include  all standard semigroups of operators on group von Neumann algebras (see Example 3), the Ornstein-Uhlenbeck semigroups on ${\Bbb R}^n$, and the heat semigroups generated by the Laplace-Beltrami operator on a compact manifold with positive curvature.
 \end{remark}

\begin{remark}
The proofs of Lemma \ref{endlem1.1} and \ref{fphis} work for any bilinear form $B(\cdot,\cdot)$ satisfying (i) $B(f,f)\geq0$; (ii) $B(T_sf,T_sf)\leq T_sB(f,f)$. In particular, the proofs work for $B(f_s,g_s)=s\frac{\partial f^*_s}{\partial s}\frac{\partial g_s}{\partial s}$ and will give the estimation
\begin{eqnarray*}
&&|\tau fg^*|\\
&\leq& c \|(\int_0^\infty |\frac{\partial T_sf}{\partial s}|^2sds)^\frac12\|^\frac12_{L_1(M)} \|(\int_0^\infty T_s|\frac{\partial T_sf}{\partial s}|^2sds)^\frac12\|^\frac12_{L_1(M)}\sup_t\|T_t\int_0^t|\frac{\partial T_sg}{\partial s}|^2sds\|_{L_\infty(M)}^\frac12.
\end{eqnarray*}
The author shows in [M08] that,
\begin{eqnarray}\label{tent}
\sup_t\|T_t\int_0^t|\frac{\partial T_sg}{\partial s}|^2sds\|_{L_\infty(M)}^\frac12\leq c\|g\|_{bmo(\T)},
\end{eqnarray}
if $T_s$ is a subordinated Poisson semigroup. It would be nice if there is a less strict assumption on $T_t$ which implies (\ref{tent}).
\end{remark}

\medskip \textbf{Proof of Theorem 0.2.} Note by (\ref{JNSS}),
\begin{eqnarray}
\|\varphi\|_{BMO(\T)}=\sup_t\|T_t|\varphi-T_t\varphi|\|_\infty=\sup_f \tau (\varphi f),
\end{eqnarray}
Here the supremum takes for all $f=hT_t(g)-T_t(hT_tg)$ with $g\geq0, \|g\|_1,\|h\|_\infty\leq1$.
We need to show that such $f$'s are in $H_1^S(\T)$ with norm $\leq c$.
The inclusion $(H^S_1)^*\subset BMO(\T)$ then follows from a density  argument. The equivalence $(H_1^S)^*=bmo(\T)=BMO(\T)$ and $H_1^G(\T)=H_1^S(\T)$ follow from inequality (\ref{HGS}) and Lemma \ref{JM2}.

Fix such a $f$.
Recall $M_t=\frac1t\int_0^tT_sds$.
\begin{eqnarray}
\tau(\int_0^t T_s\Gamma(T_sf)ds)^\frac12&\leq& \tau(\int_0^t M_tT_tT_s\Gamma(T_sf)ds)^\frac12\nonumber
\\
&\leq& \tau(\int_0^t 3M_{3t}T_{t-s}\Gamma(T_sf)ds)^\frac12\nonumber \\
(\rm Lemma \ \ref{lemma})&=& \tau(3M_{3t}T_{t}|f|^2-3M_{3t}|T_tf|^2)^\frac12 \nonumber\\
&\leq& \tau(3M_{3t}T_{t}|f|^2)^\frac12\label{0t}\\
&\leq& 2\sqrt3\tau(M_{3t}T_{2t}|T_tg|^2)^\frac12\nonumber\\
({\rm  apply\ assumption }\ (ii))&\leq& 4\sqrt2\tau(M_{8t}|T_tg|^2)^\frac12\leq c\nonumber.
\end{eqnarray}

For the other part,
\begin{eqnarray*}
\tau(\int_t^\infty T_s\Gamma(T_sf)ds)^\frac12&=&\tau(\sum_{n=0}^\infty \int_{2^nt}^{2^{n+1}t}T_s\Gamma(T_sf)ds)^\frac12\\
&\leq&\sum_{n=0}^\infty \tau( \int_{2^nt}^{2^{n+1}t}T_s\Gamma(T_sf)ds)^\frac12\\
({\rm let\ } v=s-2^nt)  &\leq&\sum_{n=0}^\infty \tau(\int_0^{2^nt}T_{2^nt+v}\Gamma(T_vT_{2^nt}f)dv)^\frac12\\
 ({\rm apply}\ (\ref{0t}))&\leq& \sum_{n=0}^\infty \tau(3M_{3\cdot2^nt}T_{2^nt}|T_{2^{n-1}t}T_{2^{n-1}t}f|^2)^\frac12\\
 &\leq& \sum_{n=0}^\infty \tau(4M_{8\cdot2^{n-1}t}|T_{2^{n-1}t}T_{2^{n-1}t}f|^2)^\frac12\\
   ({\rm apply\ assumption }\ (ii))&\leq& c\sum_{n=0}^\infty  \|T_{2^{n-1}t}f\|_1.
\end{eqnarray*}

Note by assumption (i) we have
\begin{eqnarray*}\|T_{2^{n-1}t}f\|_1=\|T_{2^{n-1}t}(T_t(hT_tg)-hT_tg)\|_1\leq \frac c{2^{rn}}.
\end{eqnarray*}

Therefore, \begin{eqnarray*}
\tau(\int_0^\infty T_s\Gamma(T_sf)ds)^\frac12&\leq& c+\sum_{n=0}^\infty\frac c{2^{rn}}\leq c. \end{eqnarray*}\qed

\medskip
\begin{example} Let us illustrate the assumptions of Theorem \ref{end} for the heat semigroups $(T_t)_t$ on a weighted Riemannian manifold with a doubling measure.  Saloff-Coste's survey (see [SC10]) gives a lot of examples of $T_t$ with kernels satisfying the following upper Gaussian bounds
\begin{eqnarray}
p(t,x,y)\leq \frac 1{V(x,\sqrt t)}\exp{-\frac {d^2(x,y)}{ct}},\\
|\frac{\partial p(t,x,y)}{\partial t}|\leq \frac 1{tV(x,\sqrt t)}\exp{-\frac {d^2(x,y)}{ct}} \label{partialp}.
\end{eqnarray}
The assumptions (i), (ii) are easily verified for such $T_t$'s. (i) is obvious by (\ref{partialp}). For (ii), it is enough to check  the extreme points $f=\delta_{x_0}$. Then $T_tf=p(t,x_0,y)$ and
\begin{eqnarray*}
(\frac1t\int_{c_1t}^{c_2t}T_sds|T_tf|^2)\leq \frac 1{V^2(x,\sqrt t)}\exp{-\frac {d^2(x,y)}{ct}}
\end{eqnarray*}
which belongs to $L^\frac12$ uniformly in $t$.
\end{example}

\begin{example} Let $(O_t)$ be the Ornstein-Uhlenbeck semigroups on ${\Bbb R}^n$ with the gaussian measure $d\mu=e^{-x^2}dx$. The infinitesimal generator of $(O_t)$ is $L=\frac12\triangle-x\cdot {\partial x}=\sum_{i=1}^n \frac12\frac{\partial^2} {\partial^2 x_i}-x_i\cdot \frac{\partial} {\partial x_i}$.  $O_t$ satisfies the $\Gamma_2\geq0$ condition and the assumption (i) of Theorem \ref{end}. This can be verified by the kernel of $O_t$.    So Theorem 0.1 applies to $O_t$. It is a pity that $O_t$ does not satisfy the assumption (ii) of Theorem \ref{end} although  the following inequality holds
\begin{eqnarray}
\int_{{\Bbb R}^n}(\frac1t\int_{t}^{2t}\tilde O_sds|O_tf|^2)^\frac12d\mu\leq c\int_{{\Bbb R}^n}fd\mu,
\end{eqnarray}
for any $f\geq0$.
Here $\tilde O_t$ is the semigroup generated by $\tilde L=  \frac12\triangle-2x\cdot {\partial x}$.
\end{example}

\section{Extension to the noncommutative setting}

\setcounter{theorem}{0} \setcounter{equation}{0}
 We refer the readers to [PX03] for an introduction of noncommuative $L^p$ spaces and to [JLX06] Chapter 10 for noncommutative semigroups of operators.
Given a semigroups of operators on a semifinite von Neumann algebras $M$, all the definitions and Lemmas in Section 1 still work in the noncommutative setting except Lemma \ref{JN}. The noncommutative generalization of Theorem 0.1 and its proof are straightforward (we kept the noncommutative version in mind when writing the proof of Theorem 0.1). Let us state it as follows without a proof.

 \begin{theorem}\label{0.1non}
 Let $(T_s)_s$ be a semigroups of operators on a semifinite von Neumann algebar $(M,\tau)$ satisfying $\Gamma_2\geq0$. Then bmo$(\T)\subset (H_1^S(\T))^*$ and
\begin{eqnarray*}
|\tau ({f}g^*)|\leq c_1\|f\|^{\frac12}_{H_1^S(\T)}\|f\|^{\frac12}_{H_1^G(\T)}\|g\|_{bmo(\T)}\leq c_2\|f\|_{H_1^S(\T)}\|g\|_{bmo(\T)},
\end{eqnarray*}
for all $f,g\in \A$ with absolute constants $c_1,c_2$.\end{theorem}

 The generalization of Theorem \ref{end} takes a little more effort. Because the John-Nirenberg inequality for noncommutative martingales are not very ``nice", and the noncommutative version of  Lemma \ref{JN} is not available to us (at least by now).

\begin{theorem}
\label{0.2non} Let $(T_s)_s$ be as in theorem \ref{0.1non}. Assume that, in addition, there exist constants $c_3,c_4,r>0$ such that,

(i)  $\|(T_{t+\varepsilon t}-T_{t})f\|_1\leq c_3 \varepsilon^r\|f\|_1,$ for all $\varepsilon>0, t>0$ and $f\in L_1(M)$.

(ii) For all $t>0$ and $g\in L_1^+(M),h\in L_2(M)$,
\begin{eqnarray}\label{Lhalf2}
\tau[(M_{8t}|h(T_tg)^{\frac12}|^2)^{\frac12}]\leq c_3(\tau g)^{\frac12}(\tau |h|^2)^{\frac12}.
\end{eqnarray}
 Then
$(H_1^S)^*=bmo(\T)=BMO(\T)$ and
$H_1^G=H_1^S$.
\end{theorem}

\begin{proof}  Note, for the BMO$(\T)$ norm,
\begin{eqnarray}
\|T_t|\varphi-T_t\varphi|^2\|^{\frac12}_\infty=\sup_f |\tau (\varphi^* f)|,
\end{eqnarray}
Here the supremum takes over for all $f=h(T_tg)^\frac12-T_t(h(T_tg)^\frac12)$ with $g\geq0, \|g\|_1,\|h\|_2\leq1$. This is easily verified as follows
\begin{eqnarray*}
\|T_t|\varphi-T_t\varphi|^2\|_\infty&=&\sup_{g\geq 0,\tau g\leq1} \tau [(T_t|\varphi-T_t\varphi|^2)g]\\
&=&\sup_g \tau [|\varphi-T_t\varphi|^2(T_tg)]\\
&=&\sup_g \tau |(\varphi-T_t\varphi)(T_tg)^\frac12|^2\\
& =&\sup_{g,\tau |h|^2\leq1}|\tau[ h(T_tg)^\frac12(\varphi-T_t\varphi)^*]|\\
& =&\sup_{g,h} |\tau \varphi^*[h(T_tg)^\frac12-T_t(h(T_tg)^\frac12) ]|
\end{eqnarray*}

We need to show such $f$'s are in $H_1^S(\T)$ with norm $\leq c$. Then the inclusion $(H^S_1)^*\subset BMO(\T)$  follows from a density  argument. And the equivalence $(H_1^S)^*=bmo(\T)=BMO(\T)$ and $H_1^G(\T)=H_1^S(\T)$ follow from inequality (\ref{HGS}) and Lemma \ref{JM2}.

 Fix such a $f$.
 Recall $M_t=\frac1t\int_0^tT_sds$.
\begin{eqnarray}
\tau(\int_0^t T_s\Gamma(T_sf)ds)^\frac12&\leq& \tau(\int_0^t M_tT_tT_s\Gamma(T_sf)ds)^\frac12\nonumber\\
&\leq& \tau(\int_0^t 3M_{3t}T_{t-s}\Gamma(T_sf)ds)^\frac12\nonumber \\
({\rm\ Lemma \ref{lemma}})&=& \tau(3M_{3t}T_{t}|f|^2-3M_{3t}|T_tf|^2)^\frac12 \nonumber\\
&=& \tau(3M_{3t}T_{t}|f|^2)^\frac12\label{30t}\\
((\ref{cp}) {\rm\ and\ triangle \ inequality})&\leq& \tau[(3M_{3t}T_{t}|h(T_tg)^\frac12|^2)^\frac12]+\tau[(3M_{3t}T_{2t}|h(T_tg)^\frac12|^2)^\frac12]\nonumber\\
({\rm \ apply (\ref{Lhalf2}))}&\leq&  c\tau[(M_{8t}|h(T_tg)^\frac12|^2)^\frac12]\leq c.\nonumber
\end{eqnarray}

The rest part of the proof remains the same as that for Theorem 0.2. Note that the noncommutative $L_{\frac12}$-quasi norm is $2$-convex, we still have $\tau [(A+B)^\frac12]\leq \tau [A^{\frac12}]+\tau [B^{\frac12}]$ for $A,B\geq0$.
\end{proof}

\medskip
\begin{example} Let $G$ be a discrete group. Let $\lambda_g, g\in G$ be the translation-operator on $\ell_2(G)$ defined as
$$\lambda_g(m)(h)=m(g^{-1}h).$$
$g\mapsto\lambda_g, g\in G$ is called the left regular representation of $G$. The so called group von Neumann algebras ${\mathcal M}_G$ of $G$ is the weak* closure of the linear span of the $\lambda_g$'s in $B(\ell_2(G))$.
The canonical trace $\tau$ on ${\mathcal M}_G$ is defined as
$\tau\lambda_e=1$ and $\tau (\lambda_g)=0$ if $g\neq e$.
If $G$ is abelian, then $L^p({\mathcal M}_G)$ is the canonical $L^p$ space of functions on the dual group $\hat G$ of $G$. In particular, if $G={\Bbb Z}$, the integer group, then $\lambda_k=e^{ikt}, k\in {\Bbb Z}$ and $L^p({\mathcal M}_G)=L^p({\Bbb T})$, the function space on the unit circle.

Let $\phi$ be a scalar valued function on $G$. We say $\phi$ is {\it conditionally negative} if
$$
 \sum_{g,h}\overline{a_g}a_h\phi(g^{-1}h)\leq0
 $$
for any finitely many coefficients $a_g\in {\Bbb C}$ with $\sum_g a_g=0$.
Schoenberg's
theorem claims that all standard semigroups of operators on the group von Neumann algebras ${\mathcal M}_G$ are in the form of
$T_t(\lambda_g)=e^{-\phi(t)}\lambda_g$ with $\phi$ a real valued conditionally negative function and $\phi(e)=0$, $\phi(g)=\phi(g^{-1})$.

Let $K_{\phi}(g,h)=\frac12(\phi(g)+\phi(h)-\phi(g^{-1}h))$, the Gromov form associated with $\phi$. Then $K_{\phi}$ is a positive definite function on $G\times G$. So is $K_{\phi}^2$. It is easy to compute by the definition that
\begin{eqnarray}
\Gamma(\sum_ga_g\lambda_g)&=&\sum_{g,h}\bar a_ga_h K_\phi(g,h)\lambda_{g^{-1}h},\\
\Gamma_2(\sum_ga_g\lambda_g)&=&\sum_{g,h}\bar a_ga_h K^2_\phi(g,h)\lambda_{g^{-1}h},
\end{eqnarray}
Therefore the $\Gamma_2 \geq0$ condition is  automatically hold for such $T_t$. So Theorem \ref{0.1non} applies to all such $(T_t)_t$'s.

Let ${\Bbb R}[G]$ be the algebra of all real valued bounded functions on $G$. Then
$$\langle \sum_ga_g\delta_g,\sum_hb_h\delta_h\rangle_{\phi}=\sum_{g,h}a_ga_hK_{\phi}(g,h)$$ defines a semi-inner product on ${\Bbb R}[G]$. After quotient the null space $$N_{\phi}=\{x\in {\Bbb R}[G], \langle x,x\rangle_{\phi}=0\},$$
 ${\Bbb R}[G]/N_{\phi}$ becomes a Hilbert space. In a forthcoming article,  we are going to show that $T_t$ satisfies the assumptions of Theorem \ref{0.2non} if ${\Bbb R}[G]/N_{\phi}$ is finite dimensional.
\end{example}
%\section*{Open problems}
%TBA

\section*{Appendix----A Carleson embedding theorem}
\setcounter{theorem}{0} \setcounter{equation}{0}\setcounter{section}{0}
Let$(M,\mu)$ be a sigma-finite measure space. Assume $(T_t)$ is a standard semigroup of operators on $L_p(M)$ with infinitesimal generator $L$.  Let $\Pt=(P_t)$
be the subordinated Poisson semigroup $P_t=e^{-t\sqrt{-L}}$. Given a $f\in L_p(M)$, then $F(t)=\Pt f=P_tf$ is $L$-harmonic on $M\times (0,\infty)$ in the sense that $(\partial^2_t +L)F=0$. Let $\nu=\nu_t d\mu$ be a measure on $M\times (0,\infty)$ with $\nu_t$ an integrable function on $M$.
Viewing $P_t$'s as analogues of the mean value operators, we say that $\nu$ is a {\it Carleson measure} with respect to $L$ if
\begin{eqnarray*}
\|\nu\|_{L,\alpha}=\sup_t\|P_{ t}\int_0^{\alpha t}\nu_sds\|_\infty<\infty.
\end{eqnarray*}
{\bf Theorem.}
Suppose $(T_t)$ is a standard semigroup of operators on a sigma-finite measure space $(M,\mu)$. Assume $T_t$ satisfies the $\Gamma_2\geq0$ condition (\ref{Gamma_2}). Let $1<p<\infty$. Then
$$\|\Pt f\|_{L_p( M\times (0,\infty),\nu))}\leq c_p\|f\|_{L^p(M)}$$ for all $f\in L^p(M)$
if $$\|\nu\|_{\Pt,4}\leq c.$$

\begin{proof}
 It is clear that $\|\Pt f\|_{L^\infty(\nu)}\leq \|f\|_{L^\infty}$. Let
 $$H_1(M)=\{f\in L_1(M); \sup_{\|g\|_{BMO(\Pt)}\leq1}|\tau fg|<\infty\}.$$
 By Lemma \ref{lemma}, we have
 $$(L_\infty(M),H_1(M))_{\frac1p}=L_p(M).$$
 It is then enough to show that
$\|\Pt f\|_{L^1(\nu)}\leq c\|f\|_{H_1(M)}.$
Note \begin{eqnarray*}
\int_{M\times (0,\infty)} |\Pt f|d\nu&\leq& \sup_{\|g\|_\infty\leq1}\int_M g(\int_0^\infty (P_tf)\nu_tdt) d\mu\\
&=& \sup_{\|g\|_\infty\leq1}\int_M f(\int_0^\infty P_t(g\nu_t)dt) d\mu\\
&\leq & \|f\|_{H_1(M)}\sup_{\|g\|_\infty\leq1}\|\int_0^\infty P_t(g\nu_t)dt\|_{BMO(\Pt)}.
\end{eqnarray*}
We go to show that $\|(\int_0^\infty P_t (g\nu_t)dt)\|_{BMO(\Pt)}\leq c\|\nu\|_{\Pt,4}$.
In fact, \begin{eqnarray*}
&&P_s|\int_0^\infty P_t (g\nu_t)dt-P_s(\int_0^\infty P_t(g\nu_t)dt)|\\
&\leq&P_s|\int_0^s (P_t-P_{t+s}) (g\nu_t)dt|+|P_s\int_s^\infty (P_t-P_{t+s}) (g\nu_t)dt)|\\
({\rm apply}\ (\ref{sbd}))&\leq&5P_s(\int_0^s  |g\nu_t|dt)+8s\int_s^\infty \frac1t P_{\frac t2} (|g\nu_t|)dt\\
(\|g\|_\infty\leq1)&\leq&5\|\nu\|_{\Pt,1}+8\sum_k\int_{2^ks}^{2^{k+1}s} \frac1{2^k} P_{\frac t2} (|\nu_t|)dt\\
&\leq&5\|\nu\|_{\Pt,1}+ \frac8{2^k}\sum_k\int_{2^ks}^{2^{k+1}s} P_{\frac t2} (|\nu_t|)dt\\
({\rm apply}\ (\ref{sbd}))&\leq&5\|\nu\|_{\Pt,1}+ \frac{16}{2^k}\sum_k\int_{2^ks}^{2^{k+1}s} P_{2^{k-1}s} (|\nu_t|)dt
\end{eqnarray*}
Taking the supremum over $s$ we get, by Lemma \ref{JN}, that
$$\|(\int_0^\infty P_t (g\nu_t)dt)\|_{BMO(\Pt)}\leq 37\|\nu\|_{\Pt,4}\leq c.$$
By interpolation, we get
$$\|\Pt f\|_{L_p(M\times(0,\infty),\nu)}\leq c_p\|f\|_{L_p(M,\mu)}.$$
\end{proof}

\textbf{Reference}

\medskip
%[AC] C. Anantharaman-Delaroche, On ergodic theorems for free group actions
%on noncommutative spaces. Probab. Theory Related Fields 135 (2006), no. 4,
%520--546.
 [A67] W.B. Arveson, Analyticity in operator algebras, Amer. J.
Math. 89(1967), 578--642.

[AMR08] P. Auscher, A. McIntosh and E. Russ, Hardy spaces of differential forms on Riemaniann
manifolds, J. Geom. Anal., 18 (2008), 192-248.

%[BL] J. Bergh, J. L\"{o}fstr\"{o}m, Interpolation spaces. An introduction.
%Grundlehren der Mathematischen Wissenschaften, No. 223. Springer-Verlag,
%Berlin-New York, 1976.

[AM] S. Avsec, T. Mei, $H_1$-BMO duality on group von Neumann algebras, preprint.

[BBG12] D. Bakry, F. Bolley, I. Gentil,
Dimension dependent hypercontractivity for Gaussian kernels, Probability Theory and Related Fields, DOI: 10.1007/s00440-011-0387, arxiv:1003.5072.

[BE85] D. Bakry and M. \'Emery, Diffusions hypercontractives. In S\'eminaire de probabilit\'es,
XIX, 1983/84, Lecture Notes in Math. 1123, pages 177–206. Springer, Berlin, 1985.

[CMS] Coifman, R. R.; Meyer, Y.; Stein, E. M. Some new function spaces and
their applications to harmonic analysis. J. Funct. Anal. 62 (1985), no. 2,
304--335.

[Da] E.B. Davies, Heat kernels and spectral theory, Cambridge Univ. Press,
1989.

[DY] X. Duong, L. Yan, Duality of Hardy and BMO spaces associated with
operators with heat kernel bounds. J. Amer. Math. Soc. 18 (2005), no. 4,
943--973.

[FS82] G. Folland and E. Stein, Hardy spaces on homogeneous groups, Princeton University Press,
Princeton, 1982.

[HM] Hofmann, Steve, Mayboroda, Svitlana,
Hardy and BMO spaces associated to divergence form elliptic operators.
Math. Ann. 344 (2009), no. 1, 37–116.

[HLMMY] S. Hofmann, G. Lu, D. Mitrea, M. Mitrea, and L. Yan, Hardy spaces associated to nonnegative
self-adjoint operators satisfying Davies-Gaffney estimates, Memoirs of the A.M.S., 214 (2011), no. 1007, vi+78 pp.

[J02] M. Junge, Doob's Inequality for Non-commutative Martingales, J. Reine
Angew. Math. 549 (2002), 149-190.

[J] M. Junge, Square function and Riesz-transform estimates for
subordinated semigroups, preprint.

[JLX06] M. Junge, C. Le Merdy, Q. Xu, $H^\infty $ Aunctional calculus and
square Aunctions on noncommutative $L^p$ -spaces. Ast\'{e}risque No. 305
(2006), vi+138 pp.

[JM10] M. Junge, T. Mei, Noncommutative Riesz Transforms-A Probabilistic
Approach,  {\it American Journal of Math.}, 2010, Vol. 132, No. 3, 611-681.

[JM12] M. Junge, T. Mei, ``BMO spaces Associated with Semigroups of
 Operators", {\it Math. Ann.}, 2012, Volume 352, Number 3, 691-743.

[JMP] M. Junge, T. Mei, J. Parcet, Smooth Fourier multipliers on group von Neumann algebras, preprint.

[JX07] M. Junge, Q. Xu, Noncommutative maximal ergodic theorems. J. Amer.
Math. Soc. 20 (2007), no. 2, 385--439.

%[Lp1] F. Lust-Piquard, Riesz transforms associated with the number operator
%on the Walsh system and the fermions. J. Funct. Anal. 155 (1998), no. 1,
%263--285.

%[Lp2] F. Lust-Piquard, Riesz transforms on deformed Fock spaces. Comm. Math.
%Phys. 205 (1999), no. 3, 519--549.

%[Lp3] F. Lust-Piquard, Dimension free estimates for discrete Riesz
%transforms on products of abelian groups. Adv. Math. 185 (2004), no. 2,
%289--327.

[L08] Li, Xiang-Dong, Martingale transforms and Lp-norm estimates of Riesz transforms on complete Riemannian manifolds. Probab. Theory Related Fields 141 (2008), no. 1-2, 247–281. 

[L09] Li, Xiang-Dong On the strong Lp-Hodge decomposition over complete Riemannian manifolds. J. Funct. Anal. 257 (2009), no. 11, 3617–3646.

[M07] T. Mei, Operator Valued Hardy Spaces, Memoirs of AMS, 2007, V. 188, No.
881.

M08]  T. Mei, ``Tent Spaces Associated with Semigroups of Operators", {\it Journal of Functional Analysis}, 255 (2008) 3356-3406.

[MTX] T. Mart\'{i}nez, T. Jos\'{i} L., Q. Xu, Vector-valued
Littlewood-Paley-Stein theory for semigroups. Adv. Math. 203 (2006), no. 2,
430--475.

[Ne74] E. Nelson, Notes on non-commutative integration, J. Funct. Anal., 15
(1974), 103-116.

[P98] G. Pisier, Non-commutative vector valued $L_ p$-spaces and completely $%
p $-summing maps. Ast\'{e}risque No. 247 (1998), vi+131 pp.

%[Mey] P. A. Meyer, D\'{e}monstration probabiliste de certaines
%in\'{e}galit\'{e}s de Littlewood-Paley. IV. Semi-groupes de convolution
%sym\'{e}triques. (French) S\'{e}minaire de Probabilit\'{e}s, X (Premi\`{e}re
%partie, Univ. Strasbourg, Strasbourg, ann\'{e}e universitaire 1974/1975),
%pp. 175--183. Lecture Notes in Math., Vol. 511, Springer, Berlin, 1976.
[PX97] G. Pisier, Q. Xu, Non-commutative Martingale Inequalities,
Comm. Math. Phys. 189 (1997), 667-698.

[PX03] G. Pisier, Q. Xu, Non-commutative $L_p$-spaces. Handbook of the
geometry of Banach spaces, Vol. 2, 1459--1517, North-Holland, Amsterdam,
2003.

[SC10] Saloff-Coste, Laurent, The heat kernel and its estimates. Probabilistic approach to geometry, 405–436,
Adv. Stud. Pure Math., 57, Math. Soc. Japan, Tokyo, 2010.

[St93] E. M. Stein, Harmonic Analysis, Princeton Univ. Press, Princeton, New
Jersey, 1993.

[St70] E. M. Stein, Topic in Harmonic Analysis (related to Littlewood-Paley
theory), Princeton Univ. Press, Princeton, New Jersey, 1970.

[V80] N. Varopoulos, Aspects of probabilistic Littlewood-Paley theory. J.
Funct. Anal. 38 (1980), no. 1, 25--60.

%\bigskip $
%\begin{array}{l}
%\mbox{Math. Dept.} \\
%\mbox{Univ. of Illinois at Urbana Champaign} \\
%\mbox{Urbana, IL, 61801} \\
%\mbox{U. S. A.} \\
%\mbox{mei@math.uiuc.edu}
%\end{array}
%$

\end{document}